\documentclass[11pt]{article}
\usepackage{amsmath, amssymb, amsthm,verbatim}

\date{}
\title{\vspace{-1cm} A new proof of the graph removal lemma}
\author{
Jacob Fox \thanks{Department of Mathematics,
Princeton University, Princeton, NJ 08544. E-mail:
jacobfox@math.princeton.edu.
Research supported by a Princeton Centennial Fellowship.}}

\newtheorem{theorem}{Theorem}
\newtheorem{lemma}{Lemma}

\begin{document}
\maketitle

\begin{abstract}
Let $H$ be a fixed graph with $h$ vertices. The graph removal lemma states that every graph on $n$ vertices with $o(n^h)$ copies of
$H$ can be made $H$-free by removing $o(n^2)$ edges. We give a new proof which avoids Szemer\'edi's regularity lemma and gives a better bound.
This approach also works to give improved bounds for the directed and multicolored analogues of the graph removal lemma. This answers questions of Alon and Gowers.
\end{abstract}

\section{Introduction}

Szemer\'edi's regularity lemma \cite{Sz} is one of the most powerful tools in graph theory.
It was introduced by Szemer\'edi in his proof \cite{Sz1} of the Erd\H{o}s-Tur\'an conjecture on
long arithmetic progressions in dense subsets of the integers. Roughly speaking, it says that every large graph
can be partitioned into a small number of parts such that the bipartite subgraph between almost every pair of parts is random-like.
This structure is useful for approximating the number of copies of some fixed subgraph.

To properly state the regularity lemma requires some terminology. The edge density $d(X,Y)$ between two subsets of vertices of a graph $G$ is the fraction of
pairs $(x,y) \in X \times Y$ that are edges of $G$. A pair $(X,Y)$ of vertex sets is called {\it $\epsilon$-regular} if for all $X' \subset X$ and
$Y' \subset Y$ with $|X'| \geq \epsilon |X|$ and $|Y'| \geq \epsilon |Y|$, we have $|d(X',Y')-d(X,Y)|<\epsilon$.
A partition $V = V_1 \cup \ldots \cup V_k$ is called equitable if $||V_i|-|V_j|| \leq 1$ for all $i$ and $j$.
The regularity lemma states that for each $\epsilon > 0$, there is a positive integer $M(\epsilon)$ such that the vertices of any graph $G$
can be equitably partitioned $V(G) = V_1 \cup \ldots \cup V_k$ into $k \leq M(\epsilon)$ parts where all but at most $\epsilon k^2$
of the pairs $(V_i,V_j)$ are $\epsilon$-regular. For more background on the regularity lemma, see the excellent survey by Koml\'os and Simonovits \cite{KoSi}.

In the regularity lemma, $M(\epsilon)$ can be taken to be a tower of twos of height proportional to $\epsilon^{-5}$. On
the other hand, Gowers \cite{Go} proved a lower bound on $M(\epsilon)$ which is a tower of twos of height proportional
to $\epsilon^{-1/16}$, thus demonstrating that $M(\epsilon)$ is inherently large as a function of $\epsilon^{-1}$. Unfortunately,
this implies that the bounds obtained by applications of the regularity lemma are usually quite poor. It remains an important
problem to determine if new proofs giving better quantitative estimates for certain applications of the regularity lemma exist (see, e.g., \cite{Go2}).
One such improvement is the proof of Gowers \cite{Go3} of Szemer\'edi's theorem using Fourier analysis.

The triangle removal lemma of Ruzsa and Szemer\'edi \cite{RuSz} is one of the most influential applications of Szemer\'edi's regularity lemma. It states
that any graph on $n$ vertices with $o(n^3)$ triangles can be made triangle-free by removing $o(n^2)$ edges. It easily implies Roth's theorem \cite{Ro}
on $3$-term arithmetic progressions in dense sets of integers. Furthermore, Solymosi \cite{So} gave an elegant proof that the triangle removal lemma further implies the
stronger corners theorem of Ajtai and Szem\'eredi \cite{AjSz}, which states that any dense subset of the integer grid contains the
vertices of an axis-aligned isosceles triangle.

The triangle removal lemma was extended by Erd\H{o}s, Frankl, and R\"odl \cite{EFR} to the graph removal lemma.
It says that for each $\epsilon>0$ and graph $H$ on $h$ vertices there is $\delta=\delta(\epsilon,H)>0$
such that every graph on $n$ vertices with at most $\delta n^{h}$ copies of $H$ can be made $H$-free by removing
at most $\epsilon n^2$ edges.  The graph removal lemma has many applications in graph theory, additive combinatorics, discrete geometry, and theoretical computer science.

One well-known application of the graph removal lemma is in property testing. This is an active area of computer science where one wishes to quickly distinguish between objects that satisfy a property from objects
that are far from satisfying that property. The study of this notion was initiated by Rubinfield and Sudan \cite{RuSu}, and subsequently Goldreich, Goldwasser, and Ron \cite{GGR} started the investigation of property
testers for combinatorial objects. One simple consequence of the graph removal lemma is a {\it constant} time algorithm for subgraph testing with one-sided error
(see \cite{Al} and its references). A graph on $n$ vertices is {\it $\epsilon$-far from being $H$-free} if at least $\epsilon n^2$ edges need
to be removed to make it $H$-free. The graph removal lemma implies that there is an algorithm which
runs in time $O_{\epsilon}(1)$ which accepts all $H$-free graphs, and rejects any graph which is $\epsilon$-far from being $H$-free
with probability at least $2/3$. The algorithm samples $t=2\delta^{-1}$ $h$-tuples of vertices uniformly at random, where
$\delta$ is picked according to the graph removal lemma, and accepts if none of them form a copy of $H$, and otherwise rejects. Any $H$-free graph
is clearly accepted. If a graph is $\epsilon$-far from being $H$-free, then it contains at least $\delta n^h$ copies of $H$,
and the probability that none of the sampled $h$-tuples forms a copy of $H$ is at most $(1-\delta)^t < 1/3$. Notice that the running time as a function of $\epsilon$ depends on the bound in the graph removal lemma.

Ruzsa and Szemer\'edi \cite{RuSz} derived the triangle removal lemma in the course of settling an extremal hypergraph problem asked by
Brown, Erd\H{o}s, and S\'os \cite{BES}. Let $g_r(n,v,e)$ be the maximum number of edges an $r$-uniform hypergraph
may have if the union of any $e$ edges span more than $v$ vertices. Ruzsa and Szemer\'edi \cite{RuSz} use the triangle removal lemma to settle the $(6,3)$-problem,
which states that $g_3(n,6,3)=o(n^2)$. Equivalently, any triple system on $n$ vertices not containing $6$ vertices with $3$ or more triples has $o(n^2)$ triples.
This was generalized by Erd\H{o}s, Frankl, and R\"odl \cite{EFR} using the graph removal lemma to establish $g_r(n,3r-3,3)=o(n^2)$.

For most of the applications of the graph removal lemma in number theory, new proofs using Fourier analysis were discovered which give better bounds (see, e.g.,
\cite{Go3}, \cite{Sh}). However, for the applications which are more combinatorial, no such methods exist. The only known proof of the graph removal lemma
used the regularity lemma, leading to weak bounds for the graph removal lemma and its applications.
Hence, finding a proof which yields better bounds by avoiding the regularity lemma is a problem of considerable interest and has been reiterated by
several authors, including Erd\H{o}s \cite{Er}, Alon \cite{Al},  Gowers \cite{Go1}, and Tao \cite{Ta1}.

Our main result is a new proof of the graph removal lemma which avoids using the regularity lemma and gives a better bound.

\begin{theorem}\label{main}
For each graph $H$ on $h$ vertices, if $\delta^{-1}$ is a tower of twos of height $5h^4\log \epsilon^{-1}$,
then every graph $G$ on $n$ vertices with at most $\delta n^h$ copies of $H$ can be made $H$-free by removing $\epsilon n^2$ edges.
\end{theorem}

For comparison, the regularity proof necessarily gives a bound on $\delta^{-1}$ that is a tower of twos of height polynomial in $\epsilon^{-1}$.

We next sketch the proof idea of the regularity lemma and our proof of the graph removal lemma. At each stage
of the proof of the regularity lemma, we have a partition $V(G)=V_1 \cup \ldots \cup V_k$ of the vertex set of a graph $G$ on $n$ vertices into parts
which differ in cardinality by at most $1$. Let $p_i=|V_i|/n$. The {\it mean square density} with respect to the partition is
$\sum_{1 \leq i,j \leq k}p_ip_j d(V_i,V_j)^2$. A {\it refinement} of a partition $\mathcal{P}$
of a set $V$ is another partition $\mathcal{Q}$ of $V$ such that each member of $\mathcal{Q}$ is a subset of some member of $\mathcal{P}$.
If the partition does not satisfy the conclusion of the regularity lemma, then using the Cauchy-Schwarz defect inequality,
the partition can be refined such that the mean square density increases by $\Omega(\epsilon^5)$ while the number of parts is at most exponential in $k$.
This process must stop after $O(\epsilon^{-5})$ steps as the mean square density cannot be more than $1$.
We thus get a bound on $M(\epsilon)$ which is a tower of twos of height $O(\epsilon^{-5})$.

Now we sketch the proof of Theorem \ref{main}. Let $H$ be a fixed graph with $h$ vertices. We suppose for contradiction that $G=(V,E)$
is a graph on $n$ vertices for which $\epsilon n^2$ edges need to be removed to make it $H$-free and yet $G$ contains less than
$\delta n^h$ copies of $H$. We pass to a subgraph $G'$ of $G$ consisting of the union of a maximum collection of edge-disjoint copies of $H$ in $G$. As the
removal of the edges of $G'$ leaves an $H$-free subgraph of $G$, the graph $G'$ has at least $\epsilon n^2$ edges. Let $d=2e(G')/n^2 \geq 2\epsilon$.
At each stage of our proof, we have a partition $V=V_1 \cup \ldots \cup V_k$ of the vertex set into parts
such that almost all vertices are in parts of the same size. Let $p_i=|V_i|/n$. The {\it mean entropy density} with respect to
the partition is $\sum_{1 \leq i,j \leq k}p_ip_j f(d(V_i,V_j))$ where $f(x)=x\log x$ for $0 < x \leq 1$ and $f(0)=0$.
A convexity argument shows that the mean entropy density with respect to any partition of $V$ is at least $d\log d$.
The fact that $f(x)$ is nonpositive for $0 \leq x \leq 1$ implies that the mean entropy density is always nonpositive.
We prove a key lemma which shows how to ``shatter'' sets with few copies of $H$, and a Jensen defect inequality for such a shattering.
These lemmas enable us to show that we can refine the partition such that the mean entropy density increases by $\Omega(d)$ while the number of
parts only goes up exponentially in $c(\epsilon,h)k$, where $c(\epsilon,h)=2^{\left(h / \epsilon\right)^{O(h^2)}}$. So essentially in each iteration the number of parts
is one exponential larger. This process must stop after $O(\log d^{-1})=O(\log \epsilon^{-1})$ steps as the mean entropy density is at least $d\log d$,
increases $\Omega(d)$ at each refinement, and is always nonpositive. We thus get a bound on $\delta^{-1}$ in the graph removal lemma which is a tower
of twos of height $O(\log \epsilon^{-1})$.

In the next section, we prove a key lemma showing how to ``shatter'' sets with few copies of $H$ between them. In Section \ref{defectsect}, we prove a Jensen defect
inequality. We use these lemmas in Section \ref{section4} to prove Theorem \ref{main}. In the concluding remarks, we discuss several variants of the graph removal
lemma for which we obtain similar improved bounds, and some open problems. We do not make any serious attempt to optimize absolute constants
in our statements and proofs. All logarithms are assumed to be base $e$.

\section{Key Lemma}

The purpose of this section is to prove a key lemma, Lemma \ref{keylemma}, for the proof of Theorem \ref{main}.
Let $H$ be a labeled graph with vertex set $[h]:=\{1,\ldots,h\}$. Lemma \ref{keylemma} shows that if $V_1,\ldots,V_h$ are vertex subsets of a graph such that there are few copies of $H$ with
the copy of vertex $i$ in $V_i$ for $i \in [h]$, then there is an edge $(i,j)$ of $H$ such that the pair $(V_i,V_j)$ can be shattered in the following sense.
An {\it $(\alpha,c,t)$-shattering} of a pair $(A,B)$ of vertex subsets in a graph $G$ is a pair of partitions
$A=A_1 \cup \ldots \cup A_r$  and $B=B_1 \cup \ldots \cup B_s$ such that $r,s \leq t$ and
the sum of $|A_i||B_j|$ over all pairs $(A_i,B_j)$ with $d(A_i,B_j)<\alpha$ is at least $c|A||B|$.
Note that if $\alpha' \geq \alpha$, $c' \leq c$, and $t' \geq t$, then an $(\alpha,c,t)$-shattering for a pair $(A,B)$
is also an $(\alpha',c',t')$-shattering for $(A,B)$. Before proving the key lemma, we first establish some
auxiliary results on $\epsilon$-regular tuples in uniform hypergraphs.

\subsection{Regular tuples in hypergraphs}

A {\it hypergraph} $\Gamma=(V,E)$ consists of a set $V$ of vertices and a set $E$ of edges, which are subsets of $V$.
A hypergraph is {\it $k$-uniform} if every edge contains precisely $k$ vertices. A $k$-uniform hypergraph $\Gamma=(V,E)$ is {\it $k$-partite}
if there is a partition $V=V_1 \cup \ldots \cup V_k$ such that every edge of $\Gamma$ contains exactly one vertex from each $V_i$. In a hypergraph $\Gamma$,
for vertex subsets $V_1,\ldots,V_k$, let $e(V_1,\ldots,V_k)$ denote the number of $k$-tuples in $V_1 \times \cdots \times V_k$ which are edges of $\Gamma$,
and let $d(V_1,\ldots,V_k)=\frac{e(V_1,\ldots,V_k)}{|V_1|\cdots |V_k|}$, which is the fraction of $k$-tuples in $V_1 \times \cdots \times V_k$
which are edges of $H$.

We begin with a simple lemma which follows by an averaging argument.

\begin{lemma}\label{averaging}
Let $\Gamma$ be a $k$-uniform hypergraph and $A_1,\ldots,A_k$ be nonempty vertex subsets. If $1 \leq a_i \leq |A_i|$ for $i \in [k]$, then
there are subsets $B_i,C_i \subset A_i$ each of cardinality $a_i$ such that $d(B_1,\ldots,B_k) \geq d(A_1,\ldots,A_k) \geq d(C_1,\ldots,C_k)$.
\end{lemma}
\begin{proof}
By averaging, the expected value of $d(X_1,\ldots,X_k)$ with $X_i \subset A_i$ chosen uniformly at random with $|X_i|=a_i$ is
$d(A_1,\ldots,A_k)$. Hence, there are choices of $B_i,C_i \subset A_i$ for each $i \in [k]$ satisfying the desired properties.
\end{proof}

In a $k$-uniform hypergraph $\Gamma$, a $k$-tuple $(V_1,\ldots,V_k)$ of vertex subsets is {\it $(\alpha,\beta)$-superregular}
if $d(U_1,\ldots,U_k) \geq \beta$ holds for all $k$-tuples $(U_1,\ldots,U_k)$ with $|U_i| \geq \alpha |V_i|$ for $i \in [k]$.

\begin{lemma}\label{presuperregular}
Suppose $\Gamma$ is a $k$-uniform hypergraph and $A_1,\ldots,A_k$ are vertex subsets each of cardinality $n$ with $d=d(A_1,\ldots,A_k)$.
If $0<\alpha,\beta<1/4$ are such that $d \geq 2\beta$ and $(A_1,\ldots,A_k)$ is not $(\alpha,\beta)$-superregular,
then there are $B_i \subset A_i$ for $i \in [k]$ with $|B_1|=\ldots=|B_k| \geq \alpha n$ and $d(B_1,\ldots,B_k) \geq (1+\frac{\alpha^k}{2})d$.
\end{lemma}
\begin{proof}
Since $(A_1,\ldots,A_k)$ is not $(\alpha,\beta)$-superregular, there are subsets $A_{i,1} \subset A_i$ such that $|A_{i,1}| \geq \alpha|A_i|$ and
$d(A_{1,1},\ldots,A_{k,1}) < \beta$. By Lemma \ref{averaging}, we may suppose that $|A_{i,1}|=\lceil \alpha n \rceil$ for $i \in [k]$. Let $A_{i,2}=A_i \setminus A_{i,1}$,
so $|A_{i,j}| \geq \alpha n$ for $i \in [k]$ and $j \in \{1,2\}$.

Summing over all $(j_1,\ldots,j_k) \in \{1,2\}^k$ with $(j_1,\ldots,j_k) \not = (1,\ldots,1)$, we have
$$\sum |A_{1,j_1}|\cdots |A_{k,j_k}| = |A_1| \cdots |A_k|-|A_{1,1}| \cdots |A_{k,1}|$$
and
\begin{eqnarray*} \sum d(A_{1,j_1},\ldots,A_{k,j_k})|A_{1,j_1}|\cdots |A_{k,j_k}| & = &\sum e(A_{1,j_1},\ldots,A_{k,j_k}) =
e(A_1,\ldots,A_k)-e(A_{1,1},\ldots,A_{k,1})\\ & = & d(A_1,\ldots,A_k)|A_1|\cdots|A_k|-d(A_{1,1},\ldots,A_{k,1})|A_{1,1}|\cdots |A_{k,1}| \\ & > &
d|A_1| \cdots |A_k|-\beta|A_{1,1}| \cdots |A_{k,1}|.\end{eqnarray*}
By averaging, there is $(j_1,\ldots,j_k) \in \{1,2\}^k$ with $(j_1,\ldots,j_k) \not = (1,\ldots,1)$ such that
\begin{eqnarray*} d(A_{1,j_1},\ldots,A_{k,j_k}) & > & \frac{d|A_1| \cdots |A_k|-\beta|A_{1,1}| \cdots |A_{k,1}|}{|A_1| \cdots |A_k|-|A_{1,1}| \cdots |A_{k,1}|}=d+(d-\beta)c/(1-c)
\geq d+(d-\beta)\alpha^k \\ & \geq & d\left(1+\frac{\alpha^k}{2}\right),\end{eqnarray*}
where $c=\frac{|A_{1,1}| \cdots |A_{k,1}|}{|A_1| \cdots |A_k|} \geq \alpha^k$. By Lemma \ref{averaging}, for each $i \in [k]$
there is a subset $B_i$ of $A_{i,j_i}$ of cardinality $\lceil \alpha n \rceil$ such that $d(B_1,\ldots,B_k) \geq d(1+\frac{\alpha^k}{2})$.
\end{proof}

The following lemma is a straightforward generalization of a result of Koml\'os that dense graphs contain large superregular pairs.

\begin{lemma}\label{superregular}
Suppose $\Gamma$ is a $k$-uniform hypergraph, and $A_1,\ldots,A_k$ are disjoint vertex subsets each of cardinality $n$.
If $0<\alpha,\beta<1/4$ are such that $d(A_1,\ldots,A_k) \geq 2\beta$, then there are subsets $V_i \subset A_i$ for $i \in [k]$
with $|V_1|=\ldots=|V_k| \geq \alpha^{3\alpha^{-k}\log \beta^{-1}}n$ for which $(V_1,\ldots,V_k)$ is $(\alpha,\beta)$-superregular.
\end{lemma}
\begin{proof}
We repeatedly apply Lemma \ref{presuperregular} until we arrive at subsets $V_i \subset A_i$ of the same size for $i \in [k]$ such that  $(V_1,\ldots,V_k)$ is
$(\alpha,\beta)$-superregular. In each application of Lemma \ref{presuperregular} we pass to subsets each with size at least an $\alpha$-fraction of the size of the
original set and the density between them is at least a factor $(1+\frac{\alpha^k}{2})$ larger than the density between the original sets. After $t$ iterations,
the density between them is at least $(1+\frac{\alpha^k}{2})^td(A_1,\ldots,A_k) \geq (1+\frac{\alpha^k}{2})^t2\beta$. This cannot continue for more than
$3\alpha^{-k}\log \beta^{-1}$ iterations since otherwise the density would be larger than $1$. Hence, we have
$|V_1|=\cdots=|V_k|\geq \alpha^{3\alpha^{-k}\log \beta^{-1}}n$, which completes the proof.
\end{proof}

The next lemma allows us to find a large matching of regular $k$-tuples.

\begin{lemma}\label{superregularmatching}
Suppose $\alpha,\beta,c,d>0$ with $\alpha,\beta<1/4$ and $d \geq 2\beta$, $\Gamma$ is a $k$-uniform hypergraph,
and $(A_1,\ldots,A_k)$ is a $(c,d)$-superregular $k$-tuple of disjoint vertex subsets each of cardinality $N$.
Then there is a positive integer $r$ such that for each $i \in [k]$
there is a partition $A_i=A_{i,0} \cup A_{i,1} \cup \ldots \cup A_{i,r}$ with $|A_{i,0}| < cN$,
and for each $j \in [r]$ the $k$-tuple $(A_{1,j},\ldots,A_{k,j})$ is $(\alpha,\beta)$-superregular
with $|A_{1,j}|=|A_{2,j}|=\cdots=|A_{k,j}| \geq \alpha^{3\alpha^{-k}\log \beta^{-1}}cN$.
\end{lemma}
\begin{proof}

In the first step, we pick out subsets $A_{i,1} \subset A_i$ for $i \in [k]$ such that the $k$-tuple $(A_{1,1},\ldots,A_{k,1})$ is
$(\alpha,\beta)$-superregular and $|A_{i,1}|=\ldots=|A_{k,1}| \geq \alpha^{3\alpha^{-k}\log \beta^{-1}}N$. We can do this by Lemma \ref{superregular}
since the $k$-tuple $(A_1,\ldots,A_k)$ is $(c,d)$-superregular and hence $d(A_1,\ldots,A_k) \geq d \geq 2\beta$.

Suppose we have already picked out $A_{i,\ell}$ for $i \in [k], \ell \in [j]$ satisfying that for each $\ell$,
$(A_{1,\ell},\ldots,A_{k,\ell})$ is $(\alpha,\beta)$-superregular, and $|A_{1,\ell}|=\cdots=|A_{k,\ell}| \geq \alpha^{3\alpha^{-k}\log \beta^{-1}}cN$.
Let $B_i=A_i \setminus \bigcup_{\ell \in j} A_{i,\ell}$, so $|B_1|=\cdots=|B_k|$. If  $|B_1| < cN$, then we let $A_{i,0}=B_i$ for $i \in [k]$
and the proof is complete. Otherwise, we pick out subsets $A_{i,j+1} \subset B_i$ for $i \in [k]$ satisfying
$$|A_{1,j+1}|=\cdots=|A_{k,j+1}| \geq \alpha^{3\alpha^{-k}\log \beta^{-1}}|B_1| \geq \alpha^{3\alpha^{-k}\log \beta^{-1}}cN$$
and $(A_{1,j+1},\ldots,A_{k,j+1})$ is $(\alpha,\beta)$-superregular. We can do this by Lemma \ref{superregular} since $(A_1,\ldots,A_k)$ is $(c,d)$-superregular,
$|B_i| \geq cN=c|A_i|$ for $i \in [k]$, and hence $d(B_1,\ldots,B_k) \geq d \geq 2\beta$. As each $A_{i,j}$ has cardinality at least
$\alpha^{3\alpha^{-k}\log \beta^{-1}}cN$, this process terminates in at most
$N/\left(\alpha^{3\alpha^{-k}\log \beta^{-1}}cN \right)=c^{-1}\alpha^{-3\alpha^{-k}\log \beta^{-1}}$ steps, and when this happens, we have the desired
partitions.
\end{proof}

\subsection{Shattering sets with few copies of $H$}\label{lastsubsect}

The following lemma is the main result of this section and is crucial for the proof of Theorem \ref{main}. Before going into the precise
statement and proof, we give a rough sketch. Let $H$ be a graph with vertex set $[h]$ and suppose $G$ is a graph with disjoint vertex sets $V_1,\ldots,V_h$
of the same size with few copies of $H$ with the copy of vertex $i$ in $V_i$ for $i \in [h]$. The lemma then says that there is an edge $(i,j)$ of $H$ for which there is an
$(\alpha,c,t)$-shattering of $(V_i,V_j)$, where $c>0$ depends only on $h$ and $t$ is not too large as a function of $\alpha$ and $h$.

The proof is by induction on $h$, with the base case $h=2$ being trivial. Let $H'$ be the induced subgraph of $H$ with vertex set $[h-1]$. The proof splits into
two cases. In the first case, there are large subsets $V_i' \subset V_i$ with few copies of $H'$ between $V_1',\ldots,V_{h-1}'$ with the copy of vertex $i$ lying
in $V_i'$. In this case, by induction, we can shatter a pair $(V_i',V_j')$ with $(i,j)$ an edge of $H'$ (and hence of $H$), and this extends to a shattering of $(V_i,V_j)$, completing this case.

In the second case, for all large subsets $V_i' \subset V_i$ there are a substantial number of copies of $H'$ between $V_1',\ldots,V_{h-1}'$ with the copy of
$i$ lying in $V_i'$. We create an auxiliary $(h-1)$-partite $(h-1)$-uniform hypergraph $\Gamma$ with parts $V_1,\ldots,V_{h-1}$ where
$(v_1,\ldots,v_{h-1}) \in V_1 \times \ldots \times V_{h-1}$ is an edge of $\Gamma$ if these vertices form a copy of $H'$ in $G$ with vertex $v_i$ the copy of $i$.
In this case we can use Lemma \ref{superregularmatching} to partition each $V_i=V_{i,0} \cup \ldots \cup \ldots \cup V_{i,z}$ with $i \in [h-1]$ such that
for each $j \in [z]$ the $(h-1)$-tuple $(V_{1,j},\ldots,V_{h-1,j})$ is $(\alpha,\beta)$-superregular in $\Gamma$ with $\beta$
not too small, $|V_{1,j}|=\ldots = |V_{h-1,j}|$ is large, and $|V_{i,0}|$ not too large. By this superregularity and the definition of $\Gamma$,
each vertex $v \in V_h$ which has for some $j$ at least $\alpha|V_{i,j}|$ neighbors in $V_{i,j}$ for each neighbor $i$ of $h$ in $H$ is a vertex of
many copies of $H$ in $G$ with the copy of $i$ in $V_i$. As there are few copies of $H$ with the copy of $i$ in $V_i$ for each $i$, this implies that
for each $j$, there are few vertices in $V_h$ which have at least $\alpha|V_{i,j}|$ neighbors in $V_{i,j}$ for each neighbor $i$ of $h$. In other words,
for most vertices $v \in V_h$ there is a neighbor $i$ of $h$ such that $v$ has less than $\alpha|V_{i,j}|$ neighbors in $V_{i,j}$.
We partition $V_h$ where a vertex $v \in V_h$ lies in a certain subset in this partition depending on
which pairs $(i,j)$ with $i$ a neighbor of $h$ in $H$ and $j \in [z]$ the vertex $v$ has less than $\alpha|V_{i,j}|$ neighbors in $V_{i,j}$. We get
that for some neighbor $i$ of $h$ in $H$, this partition of $V_h$ and the partition of $V_i$ form an $(\alpha,c,t)$-shattering of $(V_i,V_h)$.

\begin{lemma}\label{keylemma}
Let $0<\alpha<1/4$ and $d_h=2^{-(2/\alpha)^{h^2}}$. Let $H$ be a graph with vertex set $[h]$.
Suppose $G$ is a graph with disjoint vertex subsets $V_1,\ldots,V_h$ each of size $n$ such that the number of copies of $H$
with the copy of vertex $i$ in $V_i$ for $i \in [h]$ is at most $d_hn^h$.
Then there is an edge $(i,j)$ of $H$ for which there is an $(\alpha,h^{-2},2^{d_h^{-1}})$-shattering of the pair $(V_i,V_j)$.
\end{lemma}
\begin{proof}
The proof is by induction on $h$. In the base case $h=2$, as the number of edges between $V_1$ and $V_2$ is at most $d_2n^2 < \alpha n^2$,
the trivial partitions of $V_1$ and $V_2$ form an $(\alpha,1,1)$-shattering of the pair $(V_1,V_2)$. Thus the base case holds.
The induction hypothesis is that the lemma holds for $h-1$.

Let $H'$ be the induced subgraph of $H$ on vertex set $[h-1]$. Let $\Gamma$ be the $(h-1)$-partite $(h-1)$-uniform hypergraph on $V_1,\ldots,V_{h-1}$ such that $(v_1,\ldots,v_{h-1}) \in V_1 \times \ldots \times V_{h-1}$
forms an edge of $\Gamma$ if $(v_i,v_j)$ is adjacent in $G$ whenever $(i,j)$ is an edge of $H'$.

The proof splits into two cases, depending on whether or not $(V_1,\ldots,V_{h-1})$ is $(1-\frac{1}{h},d_{h-1})$-superregular in $\Gamma$.

\noindent {\bf Case 1:} $(V_1,\ldots,V_{h-1})$ is not $(1-\frac{1}{h},d_{h-1})$-superregular in $\Gamma$.
In this case, there are sets $W_i \subset V_i$ for $i \in [h-1]$ with $|W_i| \geq (1-\frac{1}{h})|V_i|$ and $d(W_1,\ldots,W_{h-1}) < d_{h-1}$.
By Lemma \ref{averaging}, letting $n' = \lceil (1-\frac{1}{h})n \rceil$,
we may suppose further that $|W_1|=\ldots=|W_{h-1}|=n'$. Therefore, the number of copies of $H'$ with the copy of vertex $i$ in $V_i$
for $i \in [h-1]$ is at most $d_{h-1}n'^{h-1}$. By the induction hypothesis, there is an edge $(i,j)$ of $H'$ (and hence also of $H$)
and partitions $W_i=A_1 \cup \ldots \cup A_{r-1}$ and $W_j=B_1 \cup \ldots \cup B_{s-1}$ with $r-1,s-1 \leq 2^{d_{h-1}^{-1}}$ and
the sum of $|A_p||B_{q}|$ over all pairs $(A_p,B_{q})$ with $d(A_p,B_{q})<\alpha$ is at least $(h-1)^{-2}|W_i||W_j|
\geq (h-1)^{-2}(1-\frac{1}{h})^2|V_i||V_j|=h^{-2}|V_i||V_j|$. Letting $A_r=V_i \setminus W_i$ and $B_s =V_j \setminus W_j$,
we have an $(\alpha,h^{-2},2^{d_{h-1}^{-1}}+1)$-shattering of the pair $(V_i,V_j)$, which completes the proof in this case.

\noindent {\bf Case 2:} $(V_1,\ldots,V_{h-1})$ is $(1-\frac{1}{h},d_{h-1})$-superregular in $\Gamma$.
In this case, by Lemma \ref{superregularmatching}, there are partitions $V_i=V_{i,0} \cup V_{i,1} \cup \ldots \cup V_{i,z}$ for $i \in [h-1]$ with
$|V_{i,0}|<(1-\frac{1}{h})|V_i|=(1-\frac{1}{h})n$ such that for each $j \in [z]$ the $(h-1)$-tuple $(V_{1,j},\ldots,V_{h-1,j})$
is $(\alpha,\beta)$-superregular in $\Gamma$ with $\beta=d_{h-1}/2$, and $|V_{1,j}|=|V_{2,j}|=\cdots=|V_{h-1,j}| \geq
\gamma n$ where $$\gamma=\alpha^{3\alpha^{1-h}\log \beta^{-1}}(1-\frac{1}{h}) > \beta^{3\alpha^{-h}} = \left(\frac{d_{h-1}}{2}\right)^{3\alpha^{-h}} >
d_{h-1}^{4\alpha^{-h}}=2^{-4\alpha^{-h}(2/\alpha)^{(h-1)^2}} \geq 2^{-(2/\alpha)^{h^2-h+1}}.$$ Since each $V_{i,j}$ has cardinality at least $\gamma n$
and each $V_i$ has cardinality $n$, we have $z \leq \frac{n}{\gamma n} =\gamma^{-1}$.

Let $I$ denote the set of neighbors of $h$ in $H$. Suppose for contradiction that there is $j \in [z]$ such that at least $|V_h|/h$ vertices $v \in V_h$ have
at least $\alpha|V_{i,j}|$ neighbors in $V_{i,j}$ for all $i \in I$. For $i \in I$, let $N(v,i)$ denote the set of neighbors of $v$ in $V_{i,j}$,
and for $i \in [h-1] \setminus I$, let
$N(v,i)=V_{i,j}$. So for at least $|V_h|/h$ vertices $v \in V_h$, $|N(v,i)| \geq \alpha|V_{i,j}|$ for $i \in [h-1]$.
Since the $(h-1)$-tuple $(V_{1,j},\ldots,V_{h-1,j})$ is $(\alpha,\beta)$-superregular in $\Gamma$, the number of copies of $H$ containing such a fixed $v$ and
with the copy of vertex $i$ in $V_{i,j}$ for $i \in [h-1]$ is at least
$$\beta|N(v,1)|\cdots|N(v,h-1)| \geq \beta\alpha^{h-1}|V_{1,j}|\cdots|V_{h-1,j}|\geq \beta \alpha^{h-1}\left(\gamma n \right)^{h-1}.$$
Hence, the number of copies of $H$ with the copy of vertex $i$ in $V_i$ for $i \in [h]$ is at least
$$\frac{|V_h|}{h}\beta \alpha^{h-1}\left(\gamma n \right)^{h-1} = h^{-1}\beta \alpha^{h-1}
\gamma^{h-1}n^h \geq (2h)^{-1}d_{h-1}\alpha^{h-1}2^{-(h-1)(2/\alpha)^{h^2-h+1}}n^h > 2^{-(2/\alpha)^{h^2}}n^h=d_h n^h,$$
contradicting that there are at most $d_h n^h$ copies of $H$ with the copy of vertex $i$ in $V_i$ for $i \in [h]$.

So, for every $j \in [z]$, less than $|V_h|/h$ vertices $v \in V_h$ have at least $\alpha|V_{i,j}|$ neighbors in $V_{i,j}$ for all
$i \in I$. For each subset $S \subset I \times [z]$, let $A_S$ denote the set of vertices $v \in V_{h}$ with less than $\alpha|V_{i,j}|$ neighbors in
$V_{i,j}$ for all $(i,j) \in S$ and at least $\alpha |V_{i,j}|$ neighbors in $V_{i,j}$ for all $(i,j) \in \left(I \times [z]\right) \setminus S$. We have
$V_h = \bigcup_{S \in I \times [z]} A_S$ is a partition of $V_h$ into $2^{|I|z}$ subsets. As for each $j \in [z]$, we have $|V_{1,j}|=\cdots=|V_{h-1,j}|$ and
more than $(1-1/h)|V_h|$ vertices in $V_h$ have less than $\alpha|V_{i,j}|$ neighbors in $V_{i,j}$ for some $i \in I$, the sum of $|A_S||V_{i,j}|$ over all
$S \subset I \times [z]$ and $i \in I$ for which $d(A_S,V_{i,j})<\alpha$ is more than
$(1-1/h)|V_h||V_{1,j}|$. Summing over all $j \in [z]$, the sum of $|A_S||V_{i,j}|$ over all
$S \subset I \times [z]$, $i \in I$, and $j \in [z]$ for which $d(A_S,V_{i,j})<\alpha$ is at least
$\sum_{j \in [z]} (1-1/h)|V_h||V_{1,j}| \geq (1-1/h)|V_h|(|V_1|/h) = (1-1/h)h^{-1}n^2$. Hence, there is $i \in I$ such that the sum of $|A_S||V_{i,j}|$
over all $S \subset I \times [z],j \in [z]$ for which $d(A_S,V_{i,j})<\alpha$ is at least $\frac{1}{|I|}(1-1/h)h^{-1}n^2 \geq h^{-2}n^2$.
As also $z+1,2^{|I|z} \leq 2^{(h-1)z} \leq 2^{d_h^{-1}}$, it follows that the partitions $V_h=\bigcup_{S \subset I \times [z]} A_S$ and
$V_i=\bigcup_{0 \leq j \leq z} V_{i,j}$ form an $(\alpha,h^{-2},2^{d_h^{-1}})$-shattering
of the pair $(V_i,V_h)$.
\end{proof}

\section{A defect inequality for convex functions}\label{defectsect}

Jensen's inequality states that if $f$ is a convex function, $\epsilon_1,\ldots,\epsilon_s$ are nonnegative real numbers which sum to $1$, and
$x_1,\ldots,x_s$ are real numbers, then $$\epsilon_1f(x_1)+\cdots+\epsilon_sf(x_s) \geq f(\epsilon_1x_1+\cdots+\epsilon_sx_s).$$
The following lemma is a simple consequence of Jensen's inequality.

\begin{lemma}\label{firstdefect}
Let $f:\mathbb{R}_{\geq 0} \to \mathbb{R}$ be a convex function, $\epsilon_1,\ldots,\epsilon_s$ and $x_1,\ldots,x_s$ be nonnegative real numbers with
$\sum_{1 \leq i \leq s}\epsilon_i=1$. For $I \subset [s]$, $c=\sum_{i \in I} \epsilon_i$ with $0<c<1$, $u=\sum_{i \in I} \epsilon_ix_i /c$, and
$v=\sum_{i \in [s]\setminus I} \epsilon_ix_i/(1-c)$, we have
$$\sum_{1 \leq i \leq s} \epsilon_if(x_i) \geq cf(u)+(1-c)f(v).$$
\end{lemma}
\begin{proof}
By Jensen's inequality, we have
$$f(u) \leq \sum_{i \in I} \frac{\epsilon_i}{c}f(x_i)$$
Since $c=\sum_{i \in I} \epsilon_i$ and $1=\sum_{1 \leq i \leq s} \epsilon_i$, then $1-c=\sum_{i \in [s] \setminus I} \epsilon_i$.
By Jensen's inequality, we have
$$f(v) \leq \sum_{i \in [s]\setminus I} \frac{\epsilon_i}{1-c}f(x_i)$$
From the two previous inequalities, we get $$cf(u)+(1-c)f(v) \leq \sum_{1 \leq i \leq s} \epsilon_if(x_i).$$
\end{proof}

Note that equality holds in Jensen's inequality when the numbers $x_1,\ldots,x_s$ are equal. A defect inequality shows that if these numbers are
far from being equal, then Jensen's inequality can be significantly improved. The following lemma
is a defect inequality for a particular convex function which we will use in the proof of Theorem \ref{main}. The lemma assumes
that a proportion $c$ of the weight is distributed amongst numbers which are at most $1/10$ of the average.

\begin{lemma}\label{secondefect}
Let $f:\mathbb{R}_{\geq 0} \to \mathbb{R}$ be the convex function given by $f(x)=x\log x$ for $x>0$ and $f(0)=0$. Let
$\epsilon_1,\ldots,\epsilon_s$, and $x_1,\ldots,x_s$ be nonnegative real numbers with $\sum_{1 \leq i \leq s}\epsilon_i=1$, and $a=\sum_{1 \leq i \leq s} \epsilon_ix_i$.
Suppose $\beta <1$ and $I \subset [s]$ is such that $x_i \leq \beta a$ for $i \in I$
and let $c=\sum_{i \in I} \epsilon_i$. Then $$\sum_{1 \leq i \leq s} \epsilon_if(x_i) \geq f(a)+(1-\beta+f(\beta))ca.$$
\end{lemma}
\begin{proof}
Notice that if $a$ or $c$ is $0$, the desired inequality is Jensen's inequality. We may therefore assume $a,c>0$. We also have $c<1$ as otherwise $c=1$,
$\epsilon_i=0$ for $i \in [s] \setminus I$, and $a=\sum_{1\leq i \leq s} \epsilon_ix_i = \sum_{i \in I} \epsilon_ix_i \leq \beta a$ as $x_i \leq \beta a$
for $i \in I$, a contradiction.  Let $u=\sum_{i \in I} \epsilon_ix_i/c$, which is a weighted average of the $x_i$ with $i \in I$, and $v=\sum_{i \in [s] \setminus I} \epsilon_ix_i/(1-c)$.
Let $\delta=u/a$, so $\delta \leq \beta$, and $\delta'=v/a=(1-\delta c)/(1-c)=1+\frac{(1-\delta)c}{1-c}$.
Note also that $cu=ca\delta$, $(1-c)v=a(1-c)\delta'$, and $cu+(1-c)v=a$. Hence, by Lemma \ref{firstdefect}, we have
\begin{eqnarray*}\sum_{1 \leq i \leq s} \epsilon_if(x_i) & \geq & cf(u)+(1-c)f(v) = f(a)+caf(\delta)+a(1-c)f(\delta') \\ & \geq &
f(a)+caf(\delta)+a(1-c)\frac{(1-\delta)c}{1-c} =  f(a)+\left(f(\delta)+1-\delta\right)ca,
\end{eqnarray*}
where the first equality follows from substituting in $f(x)=x\log x$ for $0<x \leq 1$ and $f(0)=0$, and the second inequality follows from
substituting $x=\delta'$ into the inequality $f(x) \geq x-1$ for $x \geq 0$. Since $0 \leq \delta \leq \beta < 1$,  and $f(x)+1-x$ is a decreasing function 
on the interval $[0,1]$, we get the desired inequality. 
\end{proof}

\section{Proof of Theorem \ref{main}}\label{section4}

In this section we prove Theorem \ref{main}. Our presentation is chosen to elucidate the similarities and differences
with the well known proof of Szemer\'edi's regularity lemma.

Let $G=(V,E)$ be a graph. Recall that for vertex subsets $A$ and $B$, $e(A,B)$ denotes the number of pairs $(a,b) \in A \times B$ that are edges of $G$ and
$d(A,B)=\frac{e(A,B)}{|A||B|}$ is the density of the pair $(A,B)$, which is the fraction of pairs $(a,b) \in A \times B$ that are edges of $G$.
For a function $f:\mathbb{R}_{\geq 0} \to \mathbb{R}$ define $$f(A,B)=\frac{|A||B|}{|V|^2}f(d(A,B)).$$ For partitions $\mathcal{A}$ of $A$ and $\mathcal{B}$ of $B$, let
$$f(\mathcal{A},\mathcal{B})=\sum_{A' \in \mathcal{A},B' \in \mathcal{B}}f(A',B')$$
and $f(\mathcal{A})=f(\mathcal{A},\mathcal{A})$.

\begin{lemma}\label{usefullem}
Let $f:\mathbb{R}_{\geq 0} \to \mathbb{R}$ be a convex function, $G=(V,E)$ be a graph, and $d=d(V,V)=2|E|/|V|^2$.
\begin{enumerate}
\item For vertex subsets $A,B \subset V$ and partitions $\mathcal{A}$ of $A$ and $\mathcal{B}$ of $B$, we have
$f(\mathcal{A},\mathcal{B}) \geq f(A,B)$.
\item If $\mathcal{P}$ is a partition of $V$, then $f(\mathcal{P}) \geq f(d)$.
\item If $\mathcal{P}$ and $\mathcal{P}'$ are partitions of $V$ and $\mathcal{P}'$ is a refinement of $\mathcal{P}$, then
$f(\mathcal{P}') \geq f(\mathcal{P})$.
\item Suppose that $A,B$ are vertex subsets with $d(A,B)\geq 10\alpha$, partitions $\mathcal{A}$ of $A$ and $\mathcal{B}$ of $B$ form an
$(\alpha,c,t)$-shattering of $(A,B)$, and $f(x)=x\log x$ for $x>0$ and $f(0)=0$. Then
$$f(\mathcal{A},\mathcal{B}) \geq f(A,B)+\frac{c}{2}\frac{e(A,B)}{|V|^2}.$$
\end{enumerate}
\end{lemma}
\begin{proof}
We have \begin{eqnarray*} f(\mathcal{A},\mathcal{B}) & = & \sum_{A' \in \mathcal{A},B' \in \mathcal{B}}f(A',B')=
\sum_{A' \in \mathcal{A},B' \in \mathcal{B}}\frac{|A'||B'|}{|V|^2}f(d(A',B')) \\ & = &
\frac{|A||B|}{|V|^2} \sum_{A' \in \mathcal{A},B' \in \mathcal{B}}\frac{|A'||B'|}{|A||B|}f(d(A',B')) \geq
\frac{|A||B|}{|V|^2} f(d(A,B)) = f(A,B).\end{eqnarray*}
where we used $\sum_{A' \in \mathcal{A},B' \in \mathcal{B}} \frac{|A'||B'|}{|A||B|}=1$ and Jensen's inequality. This establishes part 1. For part 2, note that if $\mathcal{P}$ is a partition of $V$, then by part 1 we have $$f(\mathcal{P})=f(\mathcal{P},\mathcal{P}) \geq f(V,V)=f(d).$$
Part 3 is an immediate corollary of part 1. For part 4, we use Lemma \ref{secondefect} such that for each $A' \in \mathcal{A}$ and $B' \in \mathcal{B}$,
there is an $i$ corresponding to the pair $(A',B')$ with $\epsilon_i=\frac{|A'||B'|}{|A||B|}$ and $x_i=d(A',B')$, and we let
$a=\sum_{i} \epsilon_ix_i=d(A,B)$, $\beta=1/10$, and $I$ be the set of $i$ such that $x_i \leq \alpha \leq \beta a$.
Since $\mathcal{A}$ is a partition of $A$ and $\mathcal{B}$ is a partition of $B$, the sum of all $\epsilon_i$ is $1$.
By the definition of an $(\alpha,c,t)$-shattering, we have $\sum_{i \in I} \epsilon_i \geq c$.
We conclude that
$$f(\mathcal{A},\mathcal{B})=\frac{|A||B|}{|V|^2}\sum_i \epsilon_if(x_i) \geq \frac{|A||B|}{|V|^2}\left(f(a)+ca(1-\beta+f(\beta))\right)  \geq f(A,B)+\frac{c}{2}\frac{e(A,B)}{|V|^2}.$$
\end{proof}

The next lemma shows how to refine a partition into not too many parts so that almost all vertices are in parts of the same size, and the remaining vertices are in parts of
smaller size.

\begin{lemma}\label{sillylemma}
Suppose $\mathcal{Q}$ is a partition of a set $V$ of size $n$ into at most $k$ parts and $\upsilon>0$. Then there is a refinement $\mathcal{Q}'$ of $\mathcal{Q}$
into at most $(2\upsilon^{-1}+1)k$ parts and a number $r$ such that all parts have size at most $r$, and all but at most $\upsilon n$ vertices are in parts of
size $r$.
\end{lemma}
\begin{proof}
If $k>\upsilon n$, then let $r=1$ and $\mathcal{Q}'$ be the partition of $V$ into parts of size $1$. Otherwise,
let $r=\lfloor \upsilon n/k \rfloor$. Refine each part in $\mathcal{Q}$ into parts of size $r$,
with possibly one remaining part of size less than $r$. The number of parts is at most $n/r+k \leq (2\upsilon^{-1}+1)k$.
The number of vertices in parts of size less than $r$ is at most $kr \leq \upsilon n$.
\end{proof}

The next lemma allows us to refine a vertex partition of a graph with many edge-disjoint copies of $H$ but with relatively few (total) copies of $H$ so that the
mean entropy density increases significantly, while the number of parts is roughly one exponential larger.

\begin{lemma}\label{almostlast}
Let $H$ be a graph on $h$ vertices. Suppose $G=(V,E)$ is a graph on $n$ vertices, whose edge set can be partitioned into $\epsilon_0 n^2$ copies of $H$.
Let  $n_0 \leq \frac{\epsilon_0}{4}n$ be a positive integer and $\mathcal{P}$ be a partition of $V$
into at most $T$ parts with all parts of size at most $n_0$, and all but at most $\frac{\epsilon_0}{8} n$ vertices in parts of size $n_0$.
Suppose further that $G$ has at most $2^{-(40/\epsilon_0)^{h^2}}T^{-h}n^h$ copies of $H$. Let $f(x)=x\log x$ for $x>0$ and $f(0)=0$.
Then there is a refinement $\mathcal{P}'$ of $\mathcal{P}$ with at most $s^T$ parts with $s=2^{2^{\left( 50/\epsilon_0 \right)^{h^2}}}$,
such that $f(\mathcal{P}') \geq f(\mathcal{P})+\frac{\epsilon_0}{4h^2}$ and all but at most $\frac{\epsilon_0}{8} n$ vertices are in parts of equal size,
and all other parts are of smaller size.
\end{lemma}
\begin{proof}
We refine the partition $\mathcal{P}$ as follows. Let $\alpha=\epsilon_0/20$, $c=h^{-2}$, and $t=2^{2^{(2/\alpha)^{h^2}}}$.
For every pair $P_i,P_j \in \mathcal{P}$ of distinct parts each of size $n_0$ for which there is an $(\alpha,c,t)$-shattering of
$(P_i,P_j)$, let $\mathcal{P}_{ij}$ denote the partition of $P_i$ and $\mathcal{P}_{ji}$ denote the partition of $P_j$
in an $(\alpha,c,t)$-shattering of the pair $(P_i,P_j)$. For each $i$, let $\mathcal{P}_i$ be the partition of $P_i$ which is the
common refinement of all partitions $\mathcal{P}_{ij}$, so $\mathcal{P}_i$ has at most $t^T$ parts. Let $\mathcal{Q}$ be the partition of $V$
consisting of all parts of the partitions $\mathcal{P}_i$. As each of the at most $T$ parts of $\mathcal{P}$ is refined into at most $t^T$ parts,
the number of parts of $\mathcal{Q}$ is at most $Tt^T$.

Let $G'$ be the subgraph of $G$ obtained by deleting edges which are inside parts of $\mathcal{P}$, contain a vertex in a part of $\mathcal{P}$ of size
not equal to $n_0$, or go between parts of $\mathcal{P}$ with density less than $\epsilon_0/2$. The number of edges inside parts is at most $nn_0/2 \leq \epsilon_0 n^2/8$. As all but at most $\frac{\epsilon_0}{8}n$ vertices are
in parts of size $n_0$, the number of edges containing a vertex in a part of size not equal to $n_0$ is at most $\frac{\epsilon_0}{8}n^2$.
The number of edges between parts of density less than $\epsilon_0/2$ is at most $(\epsilon_0/2)n^2/2 \leq \epsilon_0 n^2/4$.
So the number of edges of $G$ deleted to obtain $G'$ is at most $\epsilon_0 n^2/8+\epsilon_0 n^2/8+\epsilon_0 n^2/4=\epsilon_0 n^2/2$.
Hence, $G'$ contains at least $\epsilon_0n^2-\epsilon_0 n^2/2=\epsilon_0 n^2/2$ edge-disjoint copies of $H$. Each copy of $H$ in $G'$ has its
vertices in different parts each of size $n_0$, and its edges go between parts with density at least $\epsilon_0/2$. As every part of $\mathcal{P}$ has size at most $n_0$
and there are $T$ parts, $n_0 \geq n/T$. Note that the number of copies of $H$ in $G$ is at most $2^{-(40/\epsilon_0)^{h^2}}T^{-h}n^h=d_h(n/T)^h \leq d_hn_0^h$.
For each copy of $H$ in $G'$, by Lemma \ref{keylemma}, at least one of its edges goes between parts which are $(\alpha,c,t)$-shattered.
Hence, the number of edges of $G$ which are between parts of size $n_0$ with density at least $\frac{\epsilon_0}{2}=10\alpha$
between them and which are $(\alpha,c,t)$-shattered is at least the number of edge-disjoint copies of $H$ in $G'$, which is at least $\epsilon_0 n^2/2$.

By Lemma \ref{usefullem}, parts 1 and 4, we have
$$f(\mathcal{Q}) \geq f(\mathcal{P})+\sum_{(P_i,P_j)} \frac{c}{2}\frac{e(P_i,P_j)}{n^2} \geq f(\mathcal{P})+\frac{c}{2}\frac{\epsilon_0 n^2/2}{n^2}
\geq f(\mathcal{P})+c\epsilon_0/4 = f(\mathcal{P})+\frac{\epsilon_0}{4h^2},$$ where the sum is over all pairs $(P_i,P_j)$ of parts of $\mathcal{P}$ of size $n_0$
with $i<j$ and density at least $\frac{\epsilon_0}{2}=10\alpha$ between them that are $(\alpha,c,t)$-shattered.

By Lemma \ref{sillylemma} with $\upsilon=\frac{\epsilon_0}{8}$, there is a refinement $\mathcal{P}'$ of $\mathcal{Q}$ into at most
$$(2\upsilon^{-1}+1)|\mathcal{Q}| \leq (16\epsilon_0^{-1}+1)Tt^T \leq 17\epsilon_0^{-1}Tt^T \leq s^T$$ parts, such that
all but at most $\frac{\epsilon_0}{8} n$ vertices are in parts of equal size, and all other parts are of smaller size.
By Lemma \ref{usefullem}, part 3, we have $f(\mathcal{P}') \geq f(\mathcal{Q}) \geq f(\mathcal{P})+\frac{\epsilon_0}{4h^2}$, which completes the proof.
\end{proof}

We now have the necessary lemmas for the proof of Theorem \ref{main}.

\vspace{0.1cm}

\noindent {\bf Proof of Theorem \ref{main}:}
Suppose for contradiction that there is a graph $G$ on $n$ vertices with at most $\delta n^h$ copies of $H$ and
for which $\epsilon n^2$ edges need to be removed from $G$ to make it $H$-free. Let $G'$ be a subgraph of $G$ which consists of the union of a maximum
collection of edge-disjoint copies of $H$ in $G$. As the removal of the edges of $G'$ from $G$ leaves an $H$-free subgraph of $G$, the graph $G'$ has at least
$\epsilon n^2$ edges. Let $\epsilon_0 n^2$ denote the number of edge-disjoint copies of $H$ in $G'$, so $e(G')=e(H)\epsilon_0n^2$.

As there is at least one and at most $\delta n^h$ copies of $H$, we have $n \geq \delta^{-1/h}$.
Let $\mathcal{P}_0$ be an arbitrary partition $V=V_1 \cup \ldots \cup V_k$ of the vertex set of $G'$ into parts of size $n_0=\lceil \frac{\epsilon_0}{8}n \rceil$,
except possibly one remaining set of size less than $\frac{\epsilon_0}{8}n$. The number $p_0$ of parts of $\mathcal{P}_0$ is at most
$8\epsilon_0^{-1}+1 \leq 5h^2\epsilon^{-1}$. By Lemma \ref{usefullem}, part 2, we have $f(\mathcal{P}_0) \geq f(d) =d\log d$,
where $d=2e(G')/n^2 \geq 2\epsilon$. We repeatedly apply Lemma \ref{almostlast} to obtain a sequence
of partition refinements $\mathcal{P}_0,\mathcal{P}_1,\ldots$, and we let $p_i$ denote the number of parts of $\mathcal{P}_i$. Once we have the partition
$\mathcal{P}_i$, as long as $\delta \leq 2^{-(40/\epsilon_0)^{h^2}}p_i^{-h}$, we can apply Lemma \ref{almostlast} to obtain a refinement $\mathcal{P}_{i+1}$ of
$\mathcal{P}_i$. After $i$ iterations, $f(\mathcal{P}_i) \geq f(\mathcal{P}_0)+i\frac{\epsilon_0}{4h^2}$ and $p_i \leq s^{p_{i-1}}$,
where $s=2^{2^{\left( 50/\epsilon_0 \right)^{h^2}}}$. Roughly, at each iteration the number of parts is one exponential larger than in the previous iteration.
As $\delta^{-1}$ is a tower of twos of height $5h^4\log \epsilon^{-1}$, this process continues for at least $i_0:=\lceil 4h^4\log \epsilon^{-1} \rceil$ iterations.
Also using the inequalities $h^2\epsilon_0 > 2e(H)\epsilon_0 = d$ and $d \geq 2\epsilon$, we have
\begin{eqnarray*}
f(\mathcal{P}_{i_0}) & \geq & f(\mathcal{P}_0)+i_0\frac{\epsilon_0}{4h^2} \geq d\log d +\left(4h^4\log \epsilon^{-1}\right)\frac{\epsilon_0}{4h^2}
  =  d\log d + h^2\epsilon_0 \log \epsilon^{-1} > d\log (d/\epsilon) > 0,\end{eqnarray*}
which contradicts that $f$ applied to any partition is nonpositive. \qed

\section{Concluding remarks}

We gave a new proof of the graph removal lemma with an improved bound. Below we discuss improved bounds for several variants of the graph removal lemma
and finish with some open problems.

\vspace{0.1cm}
\noindent {\bf Removing homomorphisms.} There is a seemingly stronger variant of the graph removal lemma mentioned in \cite{EFR} which we refer to
as the homomorphism removal lemma. It states that for every graph $H$ on $h$ vertices and every $\epsilon>0$, there is $\delta>0$
such that if $G$ is a graph on $n$ vertices with at most $\delta n^h$ copies of $H$, then $\epsilon n^2$ edges of $G$ can be removed to obtain a graph $G'$
for which there is no homomorphism from $H$ to $G'$. It is rather straightforward to obtain this result from Szemer\'edi's regularity lemma.
However, one can further show that the $\delta$ in the homomorphism removal lemma is closely related to the $\delta$ in the graph removal lemma,
and thus Theorem \ref{main} implies a similar improved bound in the homomorphism removal lemma. The proof of this fact is quite simple, so we only
sketch it below.

Suppose $G$ is a graph on $n$ vertices which has at most $\delta n^h$ copies of $H$. A {\it homomorphic image} of a graph $H$ is a graph $F$ for which there is a surjective homomorphism
from $H$ to $F$. As each homomorphic image of $H$ has at most $|H|$ vertices, the number of homomorphic images of $H$ is finite. Notice that to remove all homomorphisms from $H$ to $G$, it suffices to remove all copies of homomorphic images of $H$ in $G$.
If there are few copies in $G$ of each homomorphic image of $H$, then by the graph removal lemma we can remove few edges and remove all homomorphisms from $H$ to $G$.
So there must be a homomorphic image $F$ of $H$ for which there are many copies of $F$ in $H$, say $cn^k$ with $c>\delta^{h^{-h}}$,
where $k$ is the number of vertices of $F$.  Let $f$ be a surjective homomorphism from $H$ to $F$, and for each vertex $i$ of $F$,
let $a_i$ denote the number of vertices of $H$ which map to vertex $i$ in $f$. The {\it blow-up} $F(a_1,\ldots,a_k)$ of $F$ is the graph obtained from $F$
by replacing each vertex $i$ by an independent set $I_i$ of order $a_i$, and a pair of vertices in different parts $I_i$ and $I_j$ are adjacent if and only if
$i$ and $j$ are adjacent in $F$. Note that $H$ is a subgraph of the blow-up $F(a_1,\ldots,a_k)$.
Let $\mathcal{S}$ denote the set of sequences $(v_1,\ldots,v_k)$ of $k$ vertices of $G$
which form a copy of $F$ with $v_i$ the copy of vertex $i$. If $A_1,\ldots,A_k$ are vertex subsets of $G$ with $|A_i|=a_i$ and
all $k$-tuples in $A_1 \times \cdots \times A_k$ belong to $\mathcal{S}$, then these vertex subsets form a copy of $F(a_1,\ldots,a_k)$ in $G$, and hence also
make a copy of $H$ in $G$. As $G$ has $cn^k$ copies of $F$, a simple convexity argument as in \cite{ES83} shows that if $c \gg n^{-1/(a_1\cdots a_k)}$, then
$\mathcal{S}$ contains at least $(1-o(1))c^{a_1\cdots a_k}n^{a_1+\cdots+a_k}=(1-o(1))c^{a_1\cdots a_k}n^{h}$ $k$-tuples of disjoint vertex subsets $(A_1,\ldots,A_k)$ with $|A_i|=a_i$ and
$A_1 \times \cdots \times A_k \subset \mathcal{S}$. Thus, $G$ contains at least
$$(1-o(1))c^{a_1\cdots a_k}n^{h} \geq (1-o(1)\delta^{(3^{1/3}/h)^{h}}n^h  \geq h! \delta n^h$$ labeled copies of $H$, where we use
$a_1 \cdots a_k \leq 3^{h/3}$ as $a_1,\ldots,a_k$ are positive integers which sum to $h$, and $c>\delta^{h^{-h}}$. This contradicts $G$
has at most $\delta n^h$ copies of $H$.

\vspace{0.1cm}
\noindent {\bf Directed, colored, and arithmetic removal lemmas.} The directed graph removal lemma, proved by Alon and Shapira \cite{AlSh}, states that for each directed graph $H$ on $h$ vertices and $\epsilon>0$ there is
$\delta=\delta(\epsilon,H)>0$ such that every directed graph $G=(V,E)$ on $n$ vertices with at most $\delta n^h$ copies of $H$ can be made $H$-free by removing
at most $\epsilon n^2$ edges. The proof of Theorem \ref{main} can be slightly modified to obtain a similar bound as in Theorem \ref{main}
for the directed graph removal lemma. The proof begins by finding a subgraph $G'$ of $G$ which is the disjoint union of $\epsilon' n^2$ copies of $H$,
with $\epsilon' \geq 2h^{-2}\epsilon$. There is a partition $V=V_1 \cup \ldots \cup V_h$ with at least $h^{-h}\epsilon' n^2$
edge-disjoint copies of $H$ with the copy of vertex $i$ in $V_i$. Indeed, in a uniform random partition into $h$ parts, each copy of $H$ has probability $h^{-h}$
that its copy of vertex $i$ lies in $V_i$ for all $i \in [h]$. We then let $G''$ be the subgraph of $G'$ which consists of the union of
these at least $2h^{-h-2}\epsilon n^2$ edge-disjoint copies of $H$. The proof of the directed graph removal lemma is then essentially
the same as the proof of Theorem \ref{main}, except we start with the partition $V=V_1 \cup \ldots \cup V_h$ and refine it further at each step.

There is also a colored graph removal lemma. For each $\epsilon>0$ and positive integer $h$, there is $\delta=\delta(\epsilon,H)>0$ such that if
 $\phi:E(H) \rightarrow [k]$ is a $k$-edge-coloring of the edges of a graph $H$ on $h$ vertices, and $\psi:E(G) \rightarrow [k]$ is a $k$-edge-coloring
of the edges of a graph $G$ on $n$ vertices such that the number of copies of $H$ with coloring $\phi$ in the coloring $\psi$ of $G$ is at most $\delta n^h$,
then one can remove all copies of $H$ with coloring $\phi$ by deleting at most $\epsilon n^2$ edges of $G$. We can also obtain a similarly improved
bound on the colored graph removal lemma, and the proof is identical to the proof of the directed graph removal lemma.

Green \cite{Gr} developed an arithmetic regularity lemma and used it to deduce an arithmetic removal lemma. It states that for each
$\epsilon>0$ and integer $m \geq 3$ there is $\delta>0$ such that if $G$ is an abelian group of order $N$, and $A_1,\ldots,A_m$ are subsets of $G$ such that there are
at most $\delta N^{m-1}$ solutions to the equation $a_1+\cdots+a_m=0$ with $a_i \in A_i$ for all $i$, then it is possible to remove at most
$\epsilon N$ elements from each set $A_i$ so as to obtain sets $A_i'$ for which there are no solutions to $a_1'+\cdots+a_m'=0$ with $a_i' \in A_i'$ for all $i$.
Like Szemer\'edi's regularity lemma, the bound on $\delta^{-1}$ grows as a tower of twos of height polynomial in $\epsilon^{-1}$. Green's proof of the
arithmetic regularity lemma relies on techniques from Fourier analysis and does not extend to nonabelian groups. Kr\'al, Serra, and Vena \cite{KSV} found a new proof of Green's removal lemma using the directed graph
removal lemma which extends to all groups. They proved that for each integer $m \geq 3$ and $\epsilon > 0$ there is $\delta>0$ such that the following holds.
Let $G$ be a group of order $N$, $A_1,\ldots,A_m$ be sets of elements of $G$, and $g$ be an arbitrary element of $G$. If the equation $x_1x_2\cdots x_m=g$
has at most $\delta N^{m-1}$ solutions with $x_i \in A_i$ for all $i$, then there are subsets $A_i' \subset A_i$ with $|A_i \setminus A_i'| \leq \epsilon N$
such that there is no solution to $x_1x_2\cdots x_m=g$ with $x_i \in A_i'$ for all $i$. Their proof relies on the removal lemma for directed cycles, and we thus
obtain a new bound for this removal lemma as well.

\vspace{0.1cm}
\noindent {\bf Further directions.} Alon \cite{Al} showed that the largest possible $\delta(\epsilon,H)$ in the graph removal lemma has a polynomial dependency on $\epsilon$ if and only if
$H$ is bipartite. For nonbipartite $H$, he showed that there is $c=c(H)>0$ such that $\delta(\epsilon,H)<(\epsilon/c)^{c\log (c/\epsilon)}$. Note that this
upper bound is far from the lower bound provided by Theorem \ref{main}, and it would be extremely interesting to close the gap. Similarly, Alon and Shapira \cite{AlSh} determined
for which directed graphs $H$ the function $\delta(\epsilon,H)$ in the directed graph removal lemma has a polynomial dependency on $\epsilon$. It is precisely when the
core of $H$, which is the smallest subgraph $K$ of $H$ for which there is a homomorphism from $H$ to $K$, is an oriented tree or a directed cycle of length $2$.
A similar bound also holds for Green's removal lemma. All of the superpolynomial lower bounds are based on variants of Behrend's construction \cite{Be} giving a large subset
of the first $n$ positive integers without a three-term arithmetic progression.

A great deal of research has gone into proving a hypergraph analogue of the removal lemma \cite{Go4}, \cite{Go5}, \cite{NaRoSc}, \cite{RoSk}, \cite{Ta},
leading to new proofs of Szemer\'edi's theorem and some of its extensions. Using a colored version of the hypergraph removal lemma, Shapira \cite{Sha} and
independently Kr\'al, Serra, and Vena \cite{KSV1} proved a conjecture of Green establishing a removal lemma for systems of linear equations. It would be
interesting to find new proofs of these results without using any version of the regularity lemma.

\medskip
\vspace{0.2cm}
\noindent {\bf Acknowledgement.}  I would like to thank Noga Alon, J\'anos Pach, and Benny Sudakov for helpful discussions.
I would particularly like to thank David Conlon for reading this paper carefully and providing many helpful comments.

\end{document}